\documentclass[11pt]{article}
\title{A Study of Fibonacci Cordial Labeling in Structured Graph Families}

\author{Sarbari Mitra, Soumya Bhoumik\\
Department of Mathematics\\
Fort Hays State University
}
\usepackage{amssymb}
\usepackage{amsmath}
\usepackage{algorithm}
\usepackage[noend]{algpseudocode}
\usepackage{amssymb}
\usepackage{array}
\usepackage{theorem}
\usepackage{graphicx}
\usepackage{cite}
\usepackage{colortbl}
\newtheorem{thrm}{Theorem}[section]

\newtheorem{lem}[thrm]{Lemma}

\usepackage{amsfonts}
\usepackage{graphicx}
\usepackage{graphics}
\usepackage{amsfonts}
\usepackage{multirow}

\usepackage{xmpmulti}
\usepackage[all,arc,curve,color,frame]{xy}

\theorembodyfont{\rm} \theoremstyle{definition}
\newtheorem{defin}[thrm]{Definition}

\newtheorem{conj}[thrm]{Conjecture}

 \newcounter{case}
\newenvironment{case}{\refstepcounter{case}
 \medskip \noindent {\bf Case \textup{\thecase.  }}}{\upshape}
 \renewcommand{\thecase}{\arabic{case}}

\newcounter{subcase}

 \renewcommand{\thesubcase}{\alph{subcase}}

\usepackage{fullpage}

\def\mod{{\rm mod\ }}

\def\GP{{\rm GP}}

\def\CH{{\rm CH}}

\newcounter{cases}
\newcounter{subcases}
\newenvironment{mycases}
  {%
    \setcounter{cases}{0}%
    \def\case
      {%
        \par\noindent
        \refstepcounter{cases}%
        \textbf{Case \thecases.}
      }%
  }
  {%
    \par
  }

\renewcommand*\thecases{\arabic{cases}}

\newenvironment{proof}{\noindent {\sc Proof}.}
                {\phantom{a} \hfill \framebox[2.2mm]{ } \bigskip}

\makeatletter
\def\BState{\State\hskip-\ALG@thistlm}
\makeatother

\providecommand{\keywords}[1]{\textbf{\textit{Keywords.}} #1}

\begin{document}

\pagestyle{plain}

\baselineskip = 1.2\normalbaselineskip

\maketitle

\begin{abstract}
A \emph{Fibonacci cordial labeling} of a graph \( G \) is an injective function \( f: V(G) \rightarrow \{F_0, F_1, \dots, F_n\} \), where \( F_i \) denotes the \( i^{\text{th}} \) Fibonacci number, such that the induced edge labeling \( f^*: E(G) \rightarrow \{0,1\} \), given by \( f^*(uv) = (f(u) + f(v)) \) $(\bmod\ 2)$, satisfies the balance condition \( |e_f(0) - e_f(1)| \le 1 \). Here, \( e_f(0) \) and \( e_f(1) \) represent the number of edges labeled 0 and 1, respectively. A graph that admits such a labeling is termed a \emph{Fibonacci cordial graph}. In this paper, we investigate the existence and construction of Fibonacci cordial labelings for several families of graphs, including \emph{Generalized Petersen graphs}, \emph{open and closed helm graphs}, \emph{joint sum graphs}, and \emph{circulant graphs of small order}. New results and examples are presented, contributing to the growing body of knowledge on graph labelings inspired by numerical sequences.
\end{abstract}


\section{Introduction}

The study of graph labeling has long captivated researchers within the field of graph theory. Among the many labeling concepts, a graph is said to be cordial if its vertices can be labeled with $0$s and $1$s such that the induced edge labels, based on the absolute difference of vertex labels, result in the number of vertices (or edges) labeled $0$ and those labeled $1$ differing by at most one. The concept of cordial labeling was introduced by Cahit in 1987 as a variation of graceful and harmonious labeling schemes \cite{Cahit1987}. Since then, numerous researchers have contributed to the study of various forms of cordial labelings. A comprehensive and periodically updated survey of graph labelings is maintained by Gallian \cite{Gallian2019}.

\begin{defin}
A function $f:V(G)\rightarrow \{0,1\}$ is said to be Cordial Labeling if the induced function $f^*:E(G)\rightarrow \{0,1\}$ defined by $$f^*(uv)=\vert f(u)-f(v)\vert $$
satisfies the conditions $\vert n_0-n_1\vert \le 1$, as well as $\vert \varepsilon_0-\varepsilon_1\vert \le 1$, where,  $n_i$ and $\varepsilon_i$ denote the number of vertices and edges labeled $i$, respectively, for $i = 0,1$.
\end{defin}

Over time, many variations of cordial labeling have been proposed. In 2013, Sridevi et al. \cite{SrideviNNN2013} showed that paths and cycles are Fibonacci divisor cordial graphs. Building on the idea of cordiality, Rokad and Ghodasara introduced Fibonacci cordial labeling \cite{RokadG2016}, where vertices are labeled with Fibonacci numbers rather than binary values.

\begin{defin}
The sequence $F_n$ of Fibonacci numbers is defined by the recurrence relation:
$$F_n = F_{n-1} + F_{n-2};\ \ F_0=0, F_1=F_2=1, $$
\end{defin}

\begin{defin}
An injective function $f:V(G)\rightarrow \{F_0,F_1,\cdots,F_n\}$ is said to be  Fibonacci cordial labeling if the induced function $f^*:E(G)\rightarrow \{0,1\}$ defined by $$f^*(uv)=(f(u)+f(v)) \ (\mod 2)$$
satisfies the condition $\vert \varepsilon_0-\varepsilon_1\vert \le 1$.
\end{defin}

\noindent Rokad and Ghodasara provided Fibonacci cordiality results for several families of graphs, including the Petersen graph, wheel graphs, shell graphs, bistars, and certain product graphs such as coronas. Further investigations by Rokad in 2017 \cite{Rokad2017} expanded these results to additional graph families. In \cite{MitraB2020}, we explored the Fibonacci cordiality of complete graphs, cycles, and corona products of the form $C_n \odot K_m$ for $m \leq 3$. This work was later extended to the generalized families $C_n \odot K_m$ and $C_n \odot \overline{K_m}$ in \cite{MitraB2021}. Recent contributions to this area can also be found in \cite{MitraPB2024, MitraB2024, MaderM2024, Sulayman2023}. 

In this paper, we investigate the Fibonacci cordiality of Generalized Petersen graphs, helm graphs, closed helm graphs, joint sum graphs, and finally, circulant graphs. Throughout, we denote \(\tilde{\varepsilon}:= \left| \varepsilon_1 - \varepsilon_0 \right| \), the absolute difference between the number of edges labeled $ 1$ and those labeled $0$.


\section{Main Results}
Watkins introduced the \emph{generalized Petersen graph} (abbreviated as GP) in \cite{Watkins1969}. For any integer \( n \ge 3 \) and \( k \in \{1, 2, \ldots, \lfloor (n-1)/2 \rfloor\} \), the graph \( \mathrm{GP}(n,k) \) is defined with the vertex set $ V = \{ u_i, v_i : 1 \le i \le n \} $ and the edge set $E = \{ u_i u_{i+k},\ u_i v_i,\ v_i v_{i+1} : 1 \le i \le n \}$, where the indices are taken modulo \( n \). According to this notation, $\GP(5,2)$ is the classic Petersen graph. In this paper we consider $\GP(n,1)$, and provide the Fibonacci cordial labeling for that family of graphs. 

\begin{thrm}
Generalized Petersen graph $\GP(n,1)$ is Fibonacci cordial. 
\end{thrm}
\begin{proof}
We prove the result by constructing a Fibonacci cordial labeling for $\GP(n,1)$, considering three cases based on the residue class of $n$ modulo 3. Let $f: V(\GP(n,1))\rightarrow\{F_0, F_1, \cdots, F_{2n}\}$ be the labeling function defined differently in the three cases below.
\begin{mycases}
\case $n=3p$

For $p \equiv k\ (\mod 4)$, we define, 

\[
p_1 =
\begin{cases}
(5p-k)/4, & \text{if } k=0,1,2\\
(5p+1)/4, & \text{if } k=3
\end{cases}
\qquad
p_2 =
\begin{cases}
2p+1-p_1, & \text{if } k=2\\
2p-p_1, & \text{if } k=0,1,3
\end{cases}
\]

\noindent Define the vertex labeling $f:V(\GP(n,1))\rightarrow\{F_0, F_1, \cdots, F_{2n}\}$ as follows: 
$$
f(u_i) =
\begin{cases}
\displaystyle F_{\lfloor (3i-1)/{2} \rfloor }, & \text{for }1\le i\le n-p_2\\
\displaystyle F_{3(p_1+p_2+i-n)}, & \text{for } n+1-p_2\le i\le n-1\\
\displaystyle F_{3(p_1+p_2)}, & \text{for } i=n \text{ and } k= 0,1,3  \\
\displaystyle F_{0}, & \text{for } i=n \text{ and } k=2 \\
\end{cases}
$$

For even $i$ with $2 \le i \le 2p_1$, set $f(v_i) = F_{3i/2}$. For $2p_1 + 1 \le i \le n$, define

$$
f(v_i) =
\begin{cases}
F_{\lfloor (n+3i-1)/2\rfloor }, & \text{for } k= 0,1,3  \\
F_{n+i-2-\lceil (n-i)/2\rceil}, & \text{for } k= 2 \\
\end{cases}
$$
\noindent Finally let $f(u_{n-p_2})= F_{k}$, then for $1\le i\le  2p_1-1$ when $i$ is odd, 
$$
f(v_i) =
\begin{cases}
F_{n+i+\ell- \lfloor i/4\rfloor }, & \text{if } k\equiv 1 \ (\mod 3)  \\
\displaystyle F_{n+i+\ell- \lceil \frac{n-i}{2} \rceil }, & \text{if } k\equiv 2 \ (\mod 3)
\end{cases}
$$
\noindent where $$
\ell =
\begin{cases}\vspace{.05 in}
\Big\lfloor \frac{3p}{8}\Big \rfloor, & \text{for } k= 0, 1\\ \vspace{.05 in}
\bigg\lfloor \frac{3(p-5)}{8}\bigg \rfloor, & \text{for } k=2\\
\bigg\lfloor \frac{3(p+1)}{8}\bigg \rfloor, & \text{for } k=3
\end{cases}
$$
\noindent
If \( \ell < 0 \), we redefine it as \( \ell := \ell + 1 \); this adjustment is applied throughout the proof—including in the other cases—whenever necessary. It is evident from the above construction that the labeling assigns even Fibonacci numbers to exactly $p_1 + p_2$ vertices of $G$. The total number of edges labeled with $0$ and $1$ are respectively given by, 
\[
\varepsilon_0 = 3n - 3p_1 - p_2, \quad \varepsilon_1 = 3p_1 + p_2,
\]
So the difference 
\[
\tilde{\varepsilon} = |\varepsilon_0 - \varepsilon_1| = |6p_1 + 2p_2 - 3n|
\]
satisfies $ \tilde{\varepsilon}  \le 1$ for all $p$, which can be easily verified. Hence, $G$ is Fibonacci cordial in this case.


\case $n=3p+1$

Assume $p \equiv k\ (\mod 4)$. Define:
\[
p_1 =
\begin{cases}
\frac{5p - k}{4}, & \text{if } k = 0,1\\
\frac{5p + 2}{4}, & \text{if } k = 2\\
\frac{5p + 1}{4}, & \text{if } k = 3
\end{cases}
\qquad
p_2 =
\begin{cases}
2p - p_1, & \text{if } k = 0,1,2\\
2p + 1 - p_1, & \text{if } k = 3
\end{cases}
\]

We define the labeling $f$ as follows:

$$
f(u_i) =
\begin{cases}
\displaystyle F_{\lfloor (3i-1)/{2} \rfloor }, & \text{for }1\le i\le n-p_2\\
\displaystyle F_{3(p_1+p_2+i-n)}, & \text{for } n+1-p_2\le i\le n-1\\
\displaystyle F_{3(p_1+p_2)}, & \text{for } i=n \text{ and } k=0,1,2 \\
\displaystyle F_{0}, & \text{for } i=n \text{ and } k=3 \\
\end{cases}
$$

$f(v_i)=F_{3i/2}$ for $i$ even and $2\le i\le 2p_1$. Now for $2p_1+1\le i\le n$,
$$
f(v_i) =
\begin{cases}
F_{n+i-\lfloor (n-i)/2\rfloor }, & \text{for } k=0,1,2 \\
F_{n+i-1-\lceil (n-i)/2\rceil}, & \text{for } k=3 \\
\end{cases}
$$
\noindent Finally let $f(u_{n-p_2})= F_{k}$, then for $1\le i\le  2p_1-1$ when $i$ is odd, 
$$
f(v_i) =
\begin{cases}
F_{n+i+\ell- \lfloor i/4\rfloor }, & \text{if } k\equiv 2 \ (\mod 3) \\
\displaystyle F_{n+i+\ell- \lceil \frac{n-i}{2} \rceil }, & \text{if } k\equiv 1 \ (\mod 3) 
\end{cases}
$$

\noindent where
\[
\ell =
\begin{cases}\vspace{.05 in}
\left\lfloor \frac{3(p + 2)}{8} \right\rfloor, & \text{if } k = 0\\ \vspace{.05 in}
\left\lfloor \frac{3(p + 1)}{8} \right\rfloor, & \text{if } k = 1\\\vspace{.05 in}
\left\lfloor \frac{3(p + 4)}{8} \right\rfloor, & \text{if } k = 2\\
\left\lfloor \frac{3(p - 1)}{8} \right\rfloor, & \text{if } k = 3
\end{cases}
\]

The reader can verify the parity distribution to confirm that the edge labels are indeed balanced, as asserted.

\case $n = 3p + 2$

Assume $p \equiv k\ (\mod 4)$. Define:
\[
p_1 =
\begin{cases}
\frac{5p}{4} + 1, & \text{if } k = 0\\
\frac{5p + 3}{4}, & \text{if } k = 1\\
\frac{5p + 2}{4}, & \text{if } k = 2\\
\frac{5p}{4}, & \text{if } k = 3
\end{cases}
\qquad
p_2 =
\begin{cases}
2p + 1 - p_1, & \text{if } k = 0,1,3\\
2p + 2 - p_1, & \text{if } k = 2
\end{cases}
\]

We define the labeling $f$ as follows:

$$
f(u_i) =
\begin{cases}
\displaystyle F_{\lfloor (3i-1)/{2} \rfloor }, & \text{for }1\le i\le n-p_2\\
\displaystyle F_{3(p_1+p_2+i-n)}, & \text{for } n+1-p_2\le i\le n-1\\
\displaystyle F_{3(p_1+p_2)}, & \text{for } i=n \text{ and } k=0,1,2 \\
\displaystyle F_{0}, & \text{for } i=n \text{ and } k=3 \\
\end{cases}
$$

$f(v_i)=F_{3i/2}$ for $i$ even and $2\le i\le 2p_1$. Now for $2p_1+1\le i\le n$,
$$
f(v_i) =
\begin{cases}
F_{n+i-\lceil (n-i)/2\rceil}, & \text{for } k=0,1,3 \\
F_{n+i-2-\lfloor (n-i)/2\rfloor}, & \text{for } k=2 \\
\end{cases}
$$
\noindent Finally let $f(u_{n-p_2})= F_{k}$, then for $1\le i\le  2p_1-1$ when $i$ is odd, 
$$
f(v_i) =
\begin{cases}
F_{n+i+\ell- \lfloor i/4\rfloor }, & \text{if } k\equiv 1 \ (\mod 3) \\
\displaystyle F_{n+i+\ell- \lceil \frac{n-i}{2} \rceil }, & \text{if } k\equiv 2 \ (\mod 3) 
\end{cases}
$$

\noindent where
\[
\ell =
\begin{cases}\vspace{.05 in}
\left\lfloor \frac{3(p + 3)}{8} \right\rfloor, & \text{if } k = 0\\\vspace{.05 in}
\left\lfloor \frac{3(p + 2)}{8} \right\rfloor, & \text{if } k = 1\\\vspace{.05 in}
\left\lfloor \frac{3(p - 3)}{8} \right\rfloor, & \text{if } k = 2\\
\left\lfloor \frac{3(p + 4)}{8} \right\rfloor, & \text{if } k = 3
\end{cases}
\]

Once again, verifying the parity distribution confirms that the edge labels are balanced, thereby completing the proof.
\end{mycases}
\end{proof}

\begin{figure}[h!]
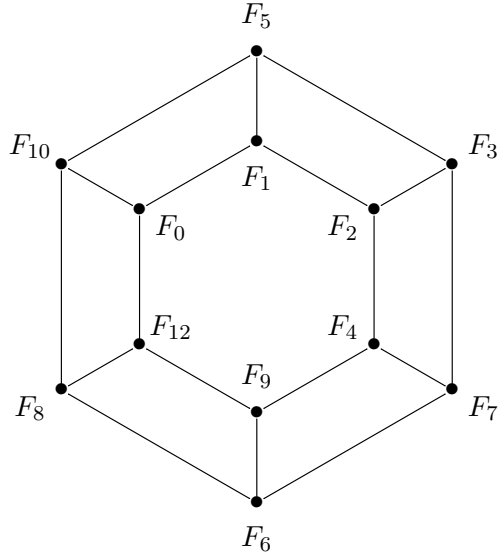

\[
\xygraph{
!{<0cm,0cm>;<0cm,1.2cm>:<-1.2cm,0cm>::}
!{(0,0);a(0)**{}?(2.5)}*{\bullet}="o0" !{(0,0);a(0)**{}?(2.9)}*{F_{5}}
!{(0,0);a(60)**{}?(2.5)}*{\bullet}="o1" !{(0,0);a(60)**{}?(2.9)}*{F_{10}}
!{(0,0);a(120)**{}?(2.5)}*{\bullet}="o2" !{(0,0);a(120)**{}?(2.9)}*{F_{8}}
!{(0,0);a(180)**{}?(2.5)}*{\bullet}="o3" !{(0,0);a(180)**{}?(2.9)}*{F_{6}}
!{(0,0);a(240)**{}?(2.5)}*{\bullet}="o4" !{(0,0);a(240)**{}?(2.9)}*{F_{7}}
!{(0,0);a(300)**{}?(2.5)}*{\bullet}="o5" !{(0,0);a(300)**{}?(2.9)}*{F_{3}}
!{(0,0);a(0)**{}?(1.5)}*{\bullet}="i0" !{(0,0);a(0)**{}?(1.1)}*{F_{1}}
!{(0,0);a(60)**{}?(1.5)}*{\bullet}="i1" !{(0,0);a(60)**{}?(1.1)}*{F_{0}}
!{(0,0);a(120)**{}?(1.5)}*{\bullet}="i2" !{(0,0);a(120)**{}?(1.1)}*{F_{12}}
!{(0,0);a(180)**{}?(1.5)}*{\bullet}="i3" !{(0,0);a(180)**{}?(1.1)}*{F_{9}}
!{(0,0);a(240)**{}?(1.5)}*{\bullet}="i4" !{(0,0);a(240)**{}?(1.1)}*{F_{4}}
!{(0,0);a(300)**{}?(1.5)}*{\bullet}="i5" !{(0,0);a(300)**{}?(1.1)}*{F_{2}}
"o0"-"o1" "o1"-"o2" "o2"-"o3" "o3"-"o4" "o4"-"o5" "o5"-"o0"
"i0"-"i1" "i1"-"i2" "i2"-"i3" "i3"-"i4" "i4"-"i5" "i5"-"i0"
"o0"-"i0" "o1"-"i1" "o2"-"i2" "o3"-"i3" "o4"-"i4" "o5"-"i5"
}
\]
\caption{Fibonacci cordial labeling of $\GP(6,1)$.}
\label{GP61}
\end{figure}

Next, we turn our attention to the \textit{helm graph} $H_n$, denoted by the graph obtained from an $n-$wheel graph by adjoining a pendant edge at each node of the cycle. We identify the vertices of $H_n$ as $V(H_n) =\{v\}\cup \{v_1,v_2,\cdots,v_n\}\cup \{u_1,u_2,\cdots,u_n\}$ where $v_i$’s are the vertices of cycle taken in clockwise and $u_i$’s are pendant vertices such that each $v_iu_i$ is a pendant edge and $v$ is the apex vertex of the cycle. Note that $\vert V(H_n)\vert =2n+1$.

\begin{thrm}
Helm graphs $H_n$ are Fibonacci cordial for $n\ge 3$. 
\end{thrm}
\begin{proof}
Let $G$ be the graph $H_n$. We identify $n=4, 8$, and $16$ as special cases and provide the individual Fibonacci labeling later. For the rest of the values of $n$, first we determine the values of three {\it parameters }$p_i$, $i=1,2,3$. First, we have provided the value for $p_3$ as follows,
$$
p_3 =
\begin{cases}
0, & \text{if } n \equiv 0 \ (\mod 4)\\
1, & \text{otherwise }
\end{cases}
$$
\noindent Now we calculuate the values of $p_1$ and $p_2$ based on three congruence classes of $n$. The following three cases are considered due to the fact that the number of available even Fibonacci numbers depends on the remainder of $n$ when divided by $3$.
\begin{mycases}
\case $n=3p$

We have $p_1=\lfloor \frac{p-1}{4}\rfloor $, $p_2=2p-p_1$.

\case $n=3p+1$, $p\ge 2$

In this case, if $p \equiv 1 \ (\mod 4)$ then $p_1=(p-9)/4$, otherwise $p_1=\big\lfloor \frac{p-1}{4}\big\rfloor $. Finally we have  $p_2=2p+2-p_1-p_3$.

\case $n=3p+2$

In this case we have $p_2=2p+2-p_1-p_3$, where
$$
p_1 =
\begin{cases}
1, & \text{for } p = 2\\
\frac{p-6}{4}, & \text{for } p \equiv 2 \ (\mod 4)\text{ and } p>2\\
\big\lfloor \frac{p+1}{4}\big \rfloor , & \text{otherwise }
\end{cases}
$$ 

\end{mycases}

Now the vertex labeling of $f: V (H_n)\rightarrow\{F_0,F_1,\cdots, F_{2n+1}\}$ is defined as follows:

$$
f(v_i) =
\begin{cases}
F_{3(i-1)}, & \text{for } 1\le i\le p_2 \\ 
F_{(i+n-p_2)+\lceil (i+n-p_2-5)/2\rceil}, & \text{for }  p_2+1 \le i\le n-2p_1 \\ 
\end{cases}
$$
\noindent For $n+1-2p_1\le i\le n-1 $
$$
f(v_i) =
\begin{cases}
F_{3(p_2+\lceil (i+2p_1-n)/2\rceil )}, & \text{if } n-i\equiv 1 (\mod 2) \text{\ and } p_3=1 \\
F_{3(p_2-1+\lceil (i+2p_1-n)/2\rceil)}, & \text{if } n-i\equiv 1 (\mod 2)  \text{\ and } p_3=0\\
F_{(3n+i)/2-p_1-p_2+\lceil (3n+i-2p_1-2p_2-10)/4\rceil )}, & \text{if } n-i\equiv 0 (\mod 2)
\end{cases}
$$

$$
f(u_i) =
\begin{cases}
F_{3p_2}, & \text{for } i=1 \text{ \ and } p_3=1 \\ 
F_{3n-p_1-p_2-1+\lceil (-p_1-p_2)/2\rceil}, & \text{for } i=1 \text{ \ and } p_3 =0\\ 
F_{i+\lceil (i-5)/2\rceil}, & \text{for } 2\le i\le n 
\end{cases}
$$

$$
f(v) =
\begin{cases}
F_{3n+1-p_1-p_2+\lceil (-p_1-p_2-1)/2\rceil}, &  \text{if } p_3 =0\\ 
F_{3n-p_1-p_2-1+\lceil (-p_1-p_2)/2\rceil}, & \text{if } p_3 =1 
\end{cases}
$$

Finally we provide Fibonacci cordial labeling of the special cases ($n=4,8,16$). Fibonacci cordial labeling for $H_4$ is provided in Figure \ref{helm}.

For $n=8,16$, we follow the following pattern: $f(u_i)=i+\lceil (i-2)/2\rceil $ for $1\le i\le n$ and 
$$
f(v_i) =
\begin{cases}
F_{3(i-1)}, & \text{for } 1\le i\le p_2\\ 
F_{i+n-p_2-1+\lceil (i+n-p_2)/2\rceil}, & \text{for }  p_2+1 \le i\le n  \\ 
F_{17}, & \text{if }  n=8  \\ 
F_{32}, & \text{if }  n=16  \\ 
\end{cases}
$$

\end{proof}

\begin{figure}[h!]
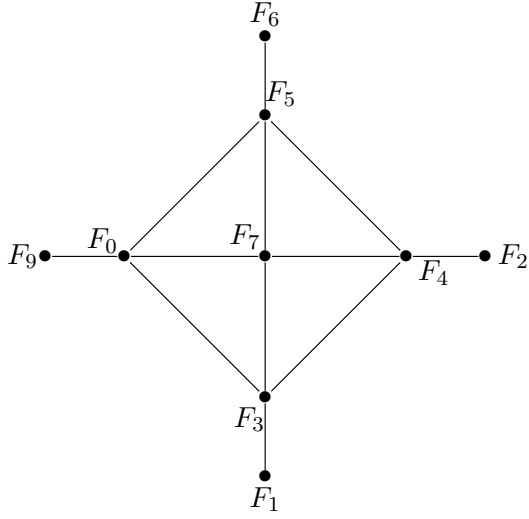

\[ \xygraph{
!{<0cm,0cm>;<0cm,0.75 cm>:<-0.75 cm,0cm>::}
!{(0,0);a(0)**{}?(0)}*{\bullet}="a" ! {(0,0);a(45)**{}?(0.5)}*{F_{7}}
!{(0,0);a(0)**{}?(2.5)}*{\bullet}="a0" !{(0,0);a(-5)**{}?(2.9)}*{F_{5}}
!{(0,0);a(90)**{}?(2.5)}*{\bullet}="a3" !{(0,0);a(85)**{}?(2.9)}*{F_{0}}
!{(0,0);a(180)**{}?(2.5)}*{\bullet}="a2" !{(0,0);a(175)**{}?(2.9)}*{F_{3}}
!{(0,0);a(270)**{}?(2.5)}*{\bullet}="a1" !{(0,0);a(265)**{}?(3)}*{F_{4}}
!{(0,0);a(0)**{}?(3.9)}*{\bullet}="b0" ! {(0,0);a(0)**{}?(4.3)}*{F_{6}}
!{(0,0);a(90)**{}?(3.9)}*{\bullet}="b3" !{(0,0);a(90)**{}?(4.3)}*{F_{9}}
!{(0,0);a(180)**{}?(3.9)}*{\bullet}="b2" !{(0,0);a(180)**{}?(4.3)}*{F_{1}}
!{(0,0);a(270)**{}?(3.9)}*{\bullet}="b1" !{(0,0);a(270)**{}?(4.4)}*{F_{2}}
"a"-"a1" "a"-"a2" "a"-"a3" "a"-"a0" 
"a0"-"a1" "a1"-"a2" "a2"-"a3" "a3"-"a0" 
"a0"-"b0" "a1"-"b1" "a2"-"b2" "a3"-"b3" 
} 
\]
\caption{Fibonacci cordial labeling of $H_{4}$ graph}
\label{helm}
\end{figure}

A closed helm $\CH_n$ (see Figure \ref{n=17FC}) is the graph obtained by taking a helm $H_n$ and adding edges between the outer pendant vertices, i.e., completing the outer circle. 
\begin{figure}[h!]
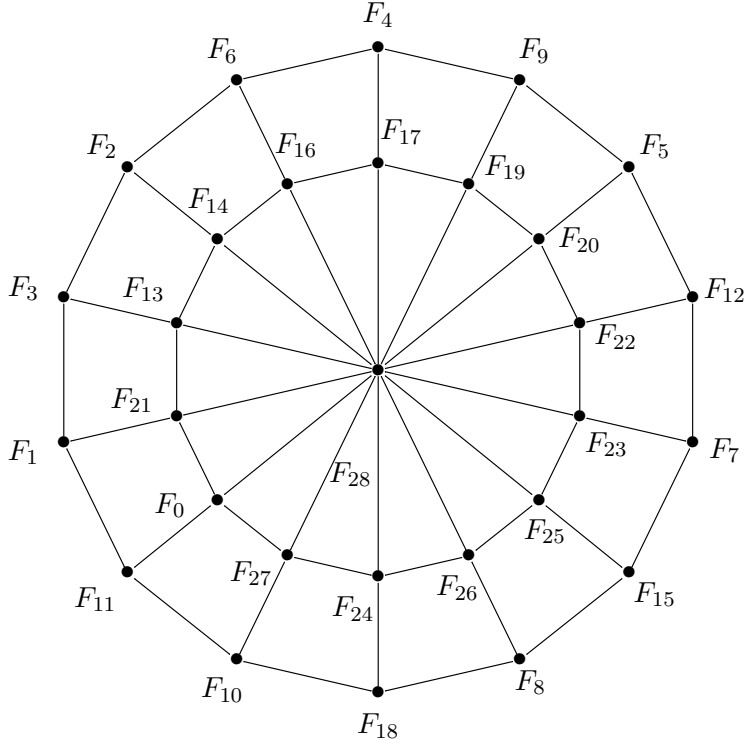

\[ \xygraph{
!{<0cm,0cm>;<0cm,1.1 cm>:<-1.1 cm,0cm>::}
!{(0,0);a(0)**{}?(0)}*{\bullet}="a14" !{(0,0);a(165)**{}?(1.3)}*{F_{28}}
!{(0,0);a(0)**{}?(2.5)}*{\bullet}="a0" !{(0,0);a(-5)**{}?(2.9)}*{F_{17}}
!{(0,0);a(26)**{}?(2.5)}*{\bullet}="a13" !{(0,0);a(20)**{}?(2.9)}*{F_{16}}
!{(0,0);a(51)**{}?(2.5)}*{\bullet}="a12" !{(0,0);a(45)**{}?(2.9)}*{F_{14}}
!{(0,0);a(77)**{}?(2.5)}*{\bullet}="a11" !{(0,0);a(71)**{}?(3)}*{F_{13}}
!{(0,0);a(103)**{}?(2.5)}*{\bullet}="a10" !{(0,0);a(97)**{}?(3)}*{F_{21}}
!{(0,0);a(129)**{}?(2.5)}*{\bullet}="a9" !{(0,0);a(123)**{}?(3)}*{F_{0}}
!{(0,0);a(154)**{}?(2.5)}*{\bullet}="a8" !{(0,0);a(148)**{}?(2.9)}*{F_{27}}
!{(0,0);a(180)**{}?(2.5)}*{\bullet}="a7" !{(0,0);a(174)**{}?(2.9)}*{F_{24}}
!{(0,0);a(206)**{}?(2.5)}*{\bullet}="a6" !{(0,0);a(200)**{}?(2.8)}*{F_{26}}
!{(0,0);a(231)**{}?(2.5)}*{\bullet}="a5" !{(0,0);a(226)**{}?(2.8)}*{F_{25}}
!{(0,0);a(257)**{}?(2.5)}*{\bullet}="a4" !{(0,0);a(252)**{}?(2.9)}*{F_{23}}
!{(0,0);a(283)**{}?(2.5)}*{\bullet}="a3" !{(0,0);a(278)**{}?(2.9)}*{F_{22}}
!{(0,0);a(309)**{}?(2.5)}*{\bullet}="a2" !{(0,0);a(303)**{}?(2.9)}*{F_{20}}
!{(0,0);a(334)**{}?(2.5)}*{\bullet}="a1" !{(0,0);a(328)**{}?(2.9)}*{F_{19}}
!{(0,0);a(0)**{}?(3.9)}*{\bullet}="b0" !{(0,0);a(0)**{}?(4.3)}*{F_{4}}
!{(0,0);a(26)**{}?(3.9)}*{\bullet}="b13" !{(0,0);a(26)**{}?(4.3)}*{F_{6}}
!{(0,0);a(51)**{}?(3.9)}*{\bullet}="b12" !{(0,0);a(51)**{}?(4.3)}*{F_{2}}
!{(0,0);a(77)**{}?(3.9)}*{\bullet}="b11" !{(0,0);a(77)**{}?(4.4)}*{F_{3}}
!{(0,0);a(103)**{}?(3.9)}*{\bullet}="b10" !{(0,0);a(103)**{}?(4.4)}*{F_{1}}
!{(0,0);a(129)**{}?(3.9)}*{\bullet}="b9" !{(0,0);a(129)**{}?(4.4)}*{F_{11}}
!{(0,0);a(154)**{}?(3.9)}*{\bullet}="b8" !{(0,0);a(154)**{}?(4.3)}*{F_{10}}
!{(0,0);a(180)**{}?(3.9)}*{\bullet}="b7" !{(0,0);a(180)**{}?(4.3)}*{F_{18}}
!{(0,0);a(206)**{}?(3.9)}*{\bullet}="b6" !{(0,0);a(206)**{}?(4.2)}*{F_{8}}
!{(0,0);a(231)**{}?(3.9)}*{\bullet}="b5" !{(0,0);a(231)**{}?(4.3)}*{F_{15}}
!{(0,0);a(257)**{}?(3.9)}*{\bullet}="b4" !{(0,0);a(257)**{}?(4.3)}*{F_{7}}
!{(0,0);a(283)**{}?(3.9)}*{\bullet}="b3" !{(0,0);a(283)**{}?(4.3)}*{F_{12}}
!{(0,0);a(309)**{}?(3.9)}*{\bullet}="b2" !{(0,0);a(309)**{}?(4.3)}*{F_{5}}
!{(0,0);a(334)**{}?(3.9)}*{\bullet}="b1" !{(0,0);a(334)**{}?(4.3)}*{F_{9}}
"a0"-"a13" "a13"-"a12" "a12"-"a11" "a11"-"a10" "a10"-"a9" "a9"-"a8" "a8"-"a7"
"a7"-"a6" "a6"-"a5" "a5"-"a4" "a4"-"a3" "a3"-"a2" "a2"-"a1"
"a1"-"a0"
"b0"-"b13" "b13"-"b12" "b12"-"b11" "b11"-"b10" "b10"-"b9" "b9"-"b8" "b8"-"b7"
"b7"-"b6" "b6"-"b5" "b5"-"b4" "b4"-"b3" "b3"-"b2" "b2"-"b1"
"b1"-"b0"
"a0"-"a14" "a1"-"a14" "a2"-"a14" "a3"-"a14" "a4"-"a14" "a5"-"a14" "a6"-"a14" "a7"-"a14" "a8"-"a14" "a9"-"a14" "a10"-"a14" "a11"-"a14" "a12"-"a14" "a13"-"a14"
"a0"-"b0" "a1"-"b1" "a2"-"b2" "a3"-"b3" "a4"-"b4" "a5"-"b5" "a6"-"b6" "a7"-"b7" 
"a8"-"b8" "a9"-"b9" "a10"-"b10" "a11"-"b11" "a12"-"b12" "a13"-"b13" 
} 
\]
\caption{Fibonacci cordial labeling of $\CH_{14}$ graph}
\label{n=17FC}
\end{figure}

\begin{thrm}
Closed Helm graphs $\CH_n$ are Fibonacci cordial for all $n\ge 3$. 
\end{thrm}
\begin{proof}
Let $G$ be the graph $\CH_n$. Due to the same reason as mentioned before, we consider three cases based on the value of $n$. Note that \( |V(\CH_n)| = 2n + 1 \) and \( |E(\CH_n)| = 4n \), which is even. Therefore, any Fibonacci cordial labeling of \( \CH_n \) must satisfy \( \varepsilon_0 = \varepsilon_1 \). The vertex labeling of $f: V (\CH_n)\rightarrow\{F_0,F_1,\cdots, F_{2n+1}\}$ is defined as follows:
\begin{mycases}
\case \textbf{$n=3p$}
$$
f(u_i) =
\begin{cases}
F_{3i/2}, & \text{for } 1\le i\le 2p_1 \text{ and } i \equiv 0 \ (\mod 2)\\
F_{\lfloor (3i+1)/4\rfloor}, & \text{for } 1\le i\le 2p_1 \text{ and } i \equiv 1 \ (\mod 2)\\
F_{\lceil 3(i-p_1)/2-1 \rceil}, & \text{for } 2p_1+1\le i\le n
\end{cases}
$$
$$
f(v_i) =
\begin{cases}
F_{3(2p_1+i+1)/2}, & \text{for } 1\le i\le 2p_2 \text{ and } i \equiv 1 \ (\mod 2)\\
F_{\lceil 3(j/2+n-p_1)/2\rceil-1}, & \text{for } 1\le i\le 2p_2 \text{ and } i \equiv 0 \ (\mod 2)\\
\displaystyle F_{n+i+p_3 + \lfloor (i+p_3+1-n)/2 \rfloor-1}, & \text{for } 2p_2+1\le i\le n-p_3\\
\displaystyle F_{[2n- 3(n-i)]}, & \text{for } n-p_3+1\le i\le n
\end{cases}
$$
\noindent where $p_2=p_3=\lceil p/2 \rceil$ and $p_1=2(p-\lceil p/2 \rceil)$, and finally $f(v)=F_{2n+1}$. It is clear that under the above labeling we get $\varepsilon_0=4n-3p_1-4p_2-2p_3$ and $\varepsilon_1=3p_1+4p_2+2p_3$. However, $p_1+p_2+p_3=2p$, hence $\tilde{\varepsilon}=\vert 6p_1+8p_2+4p_3-4n\vert = \vert 2p_1+4p_2-4p \vert = 0$, which shows that $\CH_n$ is Fibonacci cordial under the above labeling. Next, we provide the labeling for the rest of the cases for $n$, but we skip the proof as those are very similar to the present one.

\case \textbf{$n=3p+1$}

$$
f(u_i) =
\begin{cases}
F_{3i/2}, & \text{for } 1\le i\le 2p_1 \text{ and } i \equiv 0 \ (\mod 2)\\
F_{\lfloor (3i+1)/4\rfloor}, & \text{for } 1\le i\le 2p_1 \text{ and } i \equiv 1 \ (\mod 2)\\
F_{\lceil 3(i-p_1)/2-1 \rceil}, & \text{for } 2p_1+1\le i\le n
\end{cases}
$$
$$
f(v_i) =
\begin{cases}
F_{3(2p_1+i+1)/2}, & \text{for } 1\le i\le 2p_2 \text{ and } i \equiv 1 \ (\mod 2)\\
F_{\lceil 3(j/2+n-p_1)/2\rceil-1}, & \text{for } 1\le i\le 2p_2 \text{ and } i \equiv 0 \ (\mod 2)\\
\displaystyle F_{n+i+p_3 + \lfloor (i+p_3-n)/2 \rfloor-1}, & \text{for } 2p_2+1\le i\le n-p_3\\
\displaystyle F_{2n+1- 3(n-i)}, & \text{for } n-p_3+1\le i\le n
\end{cases}
$$

\noindent where $p_2=\lfloor p/2 \rfloor$, $p_3=p_2+1$, $p_1=2p+1-p_2-p_3$, and finally $f(v)=F_{2n}$.

\case \textbf{$n=3p+2$}

$$
f(u_i) =
\begin{cases}
F_{3i/2}, & \text{for } 1\le i\le 2p_1 \text{ and } i \equiv 0 \ (\mod 2)\\
F_{\lfloor (3i+1)/4\rfloor}, & \text{for } 1\le i\le 2p_1 \text{ and } i \equiv 1 \ (\mod 2)\\
F_{\lceil 3(i-p_1)/2-1 \rceil}, & \text{for } 2p_1+1\le i\le n
\end{cases}
$$
$$
f(v_i) =
\begin{cases}
F_{3(2p_1+i+1)/2}, & \text{for } 1\le i\le 2p_2 \text{ and } i \equiv 1 \ (\mod 2)\\
F_{\lceil 3(j/2+n-p_1)/2\rceil-1}, & \text{for } 1\le i\le 2p_2 \text{ and } i \equiv 0 \ (\mod 2)\\
\displaystyle F_{n+i+p_3 + \lfloor (i+p_3+1-n)/2 \rfloor-1}, & \text{for } 2p_2+1\le i\le n-p_3\\
\displaystyle F_{2n-1- 3(n-i-2)}, & \text{for } n-p_3+1\le i\le n-1\\
\displaystyle F_{0}, & \text{for } i=n
\end{cases}
$$

\noindent where $p_2 = \lfloor p/2 \rfloor$, $p_3=p_2+2$, $p_1=2(p-p_2)$, and finally $f(v)=F_{2n}$.
\end{mycases}
\end{proof}

Next, we introduce the concept of the joint sum of two graphs, as given below. 
\begin{defin}
Joint sum $G_1 \boxplus G_2$ of two graphs $G_1$ and $G_2$ is the graph obtained by connecting a vertex of $G_1$ with a vertex of $G_2$.
\end{defin}

\begin{thrm}
$C_m  \boxplus P_n$ is Fibonacci cordial.
\end{thrm}
\begin{proof}
We define the graph $C_m \boxplus P_n$ as follows :
$V(C_m \boxplus P_n)= \{ u_1, u_2, \cdots, u_m \} \cup \{ v_1, v_2, \cdots, v_n \}$, where $ \{ u_1, u_2, \cdots, u_m \} $ are the vertices of the cycle and $ \{ v_1, v_2, \cdots, v_n \} $ are the vertices of the path.  The edges of the cycle are given by  $ u_iu_{i+1 (\mod m)}$ for $ 1 \le i \le m $ 
 and $ v_jv_{j+1}$ for $ 1 \le i \le n-1 $ construct the edges of the path and the edge $ u_mv_1 $ connects the cycle and the path (as described in Figure \ref{jointsumcycle}). The following vertex labeling generates the Fibonacci Cordial Labeling:
$$
f(u_i) =
\begin{cases}
F_{0}, & \text{for }j=m\\
F_{i+1}, & \text{for } i \equiv 6,8,11 \ (\mod 12)\\
F_{i-1}, & \text{for } i \equiv 0,7,9 \ (\mod 12)\\
F_{i}, & \text{otherwise }
\end{cases}
$$ 
for all $ m $  when $ m \not \equiv 2(\mod 4)   $. Otherwise, for $ m \equiv 2(\mod 4)$,

$$
f(v_i) =
\begin{cases}
F_{i+m}, & \text{for } i \equiv (7-m), (9-m), (12-m) \ (\mod 12)\\
F_{i+m-2}, & \text{for } i \equiv (8-m), (10-m), (1-m) \ (\mod 12)\\
F_{i+m-1}, & \text{otherwise } 
\end{cases}
$$ 
\end{proof}

\begin{figure}[h!]
\[ \xygraph{
!{<0cm,0cm>;<.8 cm,0cm>:<0cm,.8 cm>::}
!{(-1,-1)}*{\bullet}="a0" !{(-0.2,-1.2) }*{F_0}
!{(-1,1)}*{\bullet}="a6" !{(-1,1.4)}*{F_5}
!{(1,1)}*{\bullet}="a5" !{(1,1.4)}*{F_7}
!{(-3,1)}*{\bullet}="a7" !{(-3,1.4)}*{F_4}
!{(3,1)}*{\bullet}="a4" !{(3,1.4)}*{F_8}
!{(-5,1)}*{\bullet}="a8" !{(-5,1.4)}*{F_{3}}
!{(5,1)}*{\bullet}="a3" !{(5,1.4)}*{F_6}
!{(-7,1)}*{\bullet}="a9" !{(-7,1.4)}*{F_2}
!{(7,1)}*{\bullet}="a2" !{(7,1.4)}*{F_9}
!{(-9,1)}*{\bullet}="a10" !{(-9,1.4)}*{F_1}
!{(-1,-2)}*{\bullet}="b1" !{(-1,-2.5) }*{F_{10}}
!{(1,-2)}*{\bullet}="b2" !{(1,-2.5) }*{F_{12}}
!{(3,-2)}*{\bullet}="b3" !{(3,-2.5) }*{F_{11}}
!{(5,-2)}*{\bullet}="b4" !{(5,-2.5) }*{F_{13}}
!{(7,-2)}*{\bullet}="b5" !{(7.4,-2.5) }*{F_{14}}
!{(7,-4)}*{\bullet}="b6" !{(7,-4.5) }*{F_{15}}
!{(5,-4)}*{\bullet}="b7" !{(5,-4.5) }*{F_{16}}
!{(3,-4)}*{\bullet}="b8" !{(3,-4.5) }*{F_{17}}
!{(1,-4)}*{\bullet}="b9" !{(1,-4.5) }*{F_{19}}
!{(-1,-4)}*{\bullet}="b10" !{(-1,-4.5) }*{F_{18}}
"a6"-"a7" "a2"-"a3" "a3"-"a4" "a4"-"a5" "a5"-"a6" "a6"-"a7"  "a7"-"a8"  "a8"-"a9"  "a9"-"a10"  "a10"-"a2" 
"a0"-"a2"
"a0"-"a10" 
"a0"-"b1" "b1"-"b2" "b2"-"b3" "b3"-"b4" "b4"-"b5" "b5"-"b6" "b6"-"b7" 
"b7"-"b8"  "b8"-"b9"  "b9"-"b10"
} \]
\caption{Fibonacci Cordial Labeling of joint sum graph $C_{10}  \boxplus P_{10}$}
\label{jointsumcycle}
\end{figure}

\begin{figure}[h!]
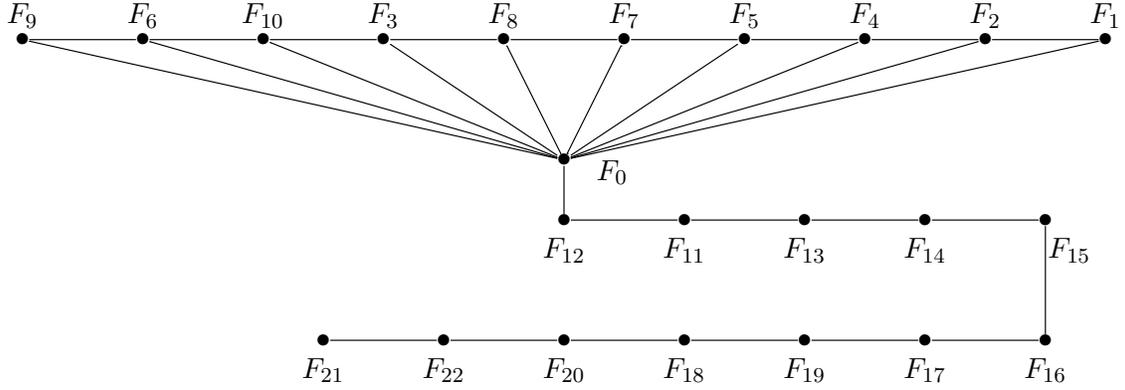

\[ \xygraph{
!{<0cm,0cm>;<.8 cm,0cm>:<0cm,.8 cm>::}
!{(0,-1)}*{\bullet}="a0" !{(0.8,-1.2) }*{F_0}
!{(-1,1)}*{\bullet}="a6" !{(-1,1.4)}*{F_8}
!{(1,1)}*{\bullet}="a5" !{(1,1.4)}*{F_7}
!{(-3,1)}*{\bullet}="a7" !{(-3,1.4)}*{F_3}
!{(3,1)}*{\bullet}="a4" !{(3,1.4)}*{F_5}
!{(-5,1)}*{\bullet}="a8" !{(-5,1.4)}*{F_{10}}
!{(5,1)}*{\bullet}="a3" !{(5,1.4)}*{F_4}
!{(-7,1)}*{\bullet}="a9" !{(-7,1.4)}*{F_6}
!{(7,1)}*{\bullet}="a2" !{(7,1.4)}*{F_2}
!{(-9,1)}*{\bullet}="a10" !{(-9,1.4)}*{F_9}
!{(9,1)}*{\bullet}="a1" !{(9,1.4)}*{F_1}
!{(0,-2)}*{\bullet}="b1" !{(0,-2.5) }*{F_{12}}
!{(2,-2)}*{\bullet}="b2" !{(2,-2.5) }*{F_{11}}
!{(4,-2)}*{\bullet}="b3" !{(4,-2.5) }*{F_{13}}
!{(6,-2)}*{\bullet}="b4" !{(6,-2.5) }*{F_{14}}
!{(8,-2)}*{\bullet}="b5" !{(8.4,-2.5) }*{F_{15}}
!{(8,-4)}*{\bullet}="b6" !{(8,-4.5) }*{F_{16}}
!{(6,-4)}*{\bullet}="b7" !{(6,-4.5) }*{F_{17}}
!{(4,-4)}*{\bullet}="b8" !{(4,-4.5) }*{F_{19}}
!{(2,-4)}*{\bullet}="b9" !{(2,-4.5) }*{F_{18}}
!{(0,-4)}*{\bullet}="b10" !{(0,-4.5) }*{F_{20}}
!{(-2,-4)}*{\bullet}="b11" !{(-2,-4.5) }*{F_{22}}
!{(-4,-4)}*{\bullet}="b12" !{(-4,-4.5) }*{F_{21}}
"a1"-"a2" "a2"-"a3" "a3"-"a4" "a4"-"a5" "a5"-"a6" "a6"-"a7" 
"a7"-"a8"  "a8"-"a9"  "a9"-"a10"  
"a0"-"a1" "a0"-"a2" "a0"-"a3" "a0"-"a4" "a0"-"a5" "a0"-"a6" "a0"-"a7" 
"a0"-"a8" "a0"-"a9" "a0"-"a10" 
"a0"-"b1" "b1"-"b2" "b2"-"b3" "b3"-"b4" "b4"-"b5" "b5"-"b6" "b6"-"b7" 
"b7"-"b8"  "b8"-"b9"  "b9"-"b10" "b10"-"b11" "b11"-"b12" 
} \]
\caption{Fibonacci cordial labeling of joint sum graph $F_{10}  \boxplus P_{12}$}
\label{jointsum}
\end{figure}

\begin{thrm}
$F_m  \boxplus P_n$ is Fibonacci cordial.
\end{thrm}
\begin{proof}
Let $\{u,u_1,u_2,\cdots,u_m\}$ and $\{v_1,v_2,\cdots,v_n\}$ be the vertices of $F_m$ and $P_n$ ($u$ being the apex vertex of $F_n$). We obtain the graph $F_m  \boxplus P_n$ by connecting $u$ and $v_1$. We define $f:V(F_m  \boxplus P_n) \rightarrow \{F_0,F_1,\cdots, F_{m+n+1}\}$ as follows: $f(u)=F_0$ and

$$
f(u_j) =
\begin{cases}
\displaystyle F_{j+\lceil \frac{j-2}{2}\rceil }, & \text{for } 1\le j\le p_1+1\\
\displaystyle F_{3(j-p_1)/2}, & \text{for } p_1+2 \le j\le p_1+2p_2 \text{\ and $j$ is of the same parity as $p_1$ }\\
\displaystyle F_{3(j-p_2-p_1)}, & \text{for } p_1+2p_2+1\le j\le m\\ 
\displaystyle F_{\lfloor \frac{3(j+p_1+3)}{2}\rfloor -2}, & \text{for } p_1+3 \le j\le p_1+2p_2-1 \text{\ and $j$ is of the opposite parity as $p_1$ }
\end{cases}
$$
\noindent where we obtain the values of $p_1$ and $p_2$ from below cases.

\case $m=6p$

$p_1=3p$, $p_2=p$, and 
$
f(v_j) =
\begin{cases}
\displaystyle F_{j+n+1}, & \text{for } j\equiv 6,9,10 \ (\mod 12)\\
\displaystyle F_{j+n-1}, & \text{for } j\equiv 7 \ (\mod 12)\\
\displaystyle F_{j+n-2}, & \text{for } j\equiv 11 \ (\mod 12)\\
\displaystyle F_{j+n}, & \text{otherwise } 
\end{cases}
$

\case $m=6p+1$

$p_1=3p+1$, $p_2=p$, and 
$
f(v_j) =
\begin{cases}
\displaystyle F_{j+n+1}, & \text{for } j\equiv 2,4,7 \ (\mod 12)\\
\displaystyle F_{j+n-1}, & \text{for } j\equiv 3,5,8 \ (\mod 12)\\
\displaystyle F_{j+n}, & \text{otherwise } 
\end{cases}
$

\case $m=6p+2$

$p_1=3p+2$, $p_2=p$, and 
$
f(v_j) =
\begin{cases}
\displaystyle F_{j+n-2}, & \text{for } j\equiv 3 \ (\mod 12)\\
\displaystyle F_{j+n-1}, & \text{for } j\equiv 11 \ (\mod 12)\\
\displaystyle F_{j+n+1}, & \text{for } j\equiv 1,2,10 \ (\mod 12)\\
\displaystyle F_{j+n}, & \text{otherwise } 
\end{cases}
$

\case $m=6p+3$

$p_1=3p+3$, $p_2=p$, and 
$
f(v_j) =
\begin{cases}
\displaystyle F_{j+n-2}, & \text{for } j\equiv 2,11 \ (\mod 12)\\
\displaystyle F_{j+n-1}, & \text{for } j\equiv 7 \ (\mod 12)\\
\displaystyle F_{j+n}, & \text{for } j\equiv 3,4,5,8 \ (\mod 12)\\
\displaystyle F_{j+n+1}, & \text{otherwise } 
\end{cases}
$

\case $m=6p+4$

$p_1=3p+2$, $p_2=p+1$, and 
$
f(v_j) =
\begin{cases}
\displaystyle F_{j+n-1}, & \text{for } j\equiv 0,3,10 \ (\mod 12)\\
\displaystyle F_{j+n+1}, & \text{for } j\equiv 1,4,11 \ (\mod 12)\\
\displaystyle F_{j+n}, & \text{otherwise } 
\end{cases}
$

\case $m=6p+5$

$p_1=3p+2$, $p_2=p+1x$, and 
$
f(v_j) =
\begin{cases}
\displaystyle F_{j+n+1}, & \text{for } j\equiv 0,3,10 \ (\mod 12)\\
\displaystyle F_{j+n-1}, & \text{for } j\equiv 1,4,11 \ (\mod 12)\\
\displaystyle F_{j+n}, & \text{otherwise } 
\end{cases}
$

\end{proof}

Finally, we turn our attention to circulant graphs, a distinguished class of Cayley graphs renowned for their rich symmetry and enduring appeal among mathematicians. Circulant networks naturally generalize double loop networks, first introduced by Wong and Coppersmith \cite{WongC1974}. For decades, circulant graphs have played a significant role in the design of computer and telecommunication networks, valued for their optimal fault-tolerance and efficient routing properties \cite{BoeschW1985}.

\begin{defin}\label{circulant}
Let $\mathbb{Z}_n$ be a cyclic group and $S\subset \mathbb{Z}_n$ such that $\{0\}\notin S$. Define a graph $\Gamma = \Gamma(n,S)$ by $V(\Gamma) = \mathbb{Z}_n$ and $E(\Gamma) =\{(u,v): v-u\in S\}$. Such a graph is a {\bf \mathversion{bold}circulant graph of order $n$} with connection set $S$. Note that $S = S^{-1} = \{-s:s\in S\}$ for circulant graphs.
\end{defin}

\noindent We focus our work mainly on ``small" connection sets. Note that $\vert S \vert \le n-1$, since $\{0\}\notin S$ (no loops). Also when $\vert S \vert = n-1$, then $\Gamma$ becomes complete graph, and it is well known that $\Gamma(n,\{1,2,\cdots,n-1\})$ is not Fibonacci cordial, except $n=4,6,7,9,11,18,22$ (see \cite{MitraB2020}). Let us denote $\tilde{\varepsilon}_f=\vert \varepsilon_f(0)-\varepsilon_f(1)\vert $ and $\tilde{\varepsilon}_g=\vert \varepsilon_g(0)-\varepsilon_g(1)\vert $. We begin this section with the following result.

\begin{lem}\label{cycle}
For any two injective Fibonacci labelings $f$ and $g$ on the circulant graph $\Gamma(n,\pm\{1, 2, \ldots, j\})$, we have $\tilde{\varepsilon}_f -\tilde{\varepsilon}_g \equiv 0 \pmod{4}$.
\end{lem}

\begin{proof}
Without loss of generality, assume that the labelings $f$ and $g$ differ only at a single vertex $v_0 \in V(\Gamma(n,\pm\{1, 2, \ldots, j\}))$, i.e., $f(v_0) \neq g(v_0)$ and $f(v) = g(v)$ for all $v \neq v_0$. Let $v_1, v_2, \ldots, v_{2j}$ denote the $2j$ neighbors of $v_0$ in the circulant graph. For the labeling $f$, suppose that $v_0$ shares the same parity with exactly $k$ of its neighbors. Since $f$ is a Fibonacci labeling, this means that $k$ edges will have induced $0$ labels, and the remaining $ 2j-k$ edges have label $1$. Thus, the contribution to $\tilde{\varepsilon}_f$ from the edges incident to $v_0$ is:
\[
\sum_{i=1}^{2j} f^*(v_0v_i) = \vert k - (2j - k) \vert = \vert 2k - 2j\vert .
\]

Similarly, since $f$ and $g$ differ only at $v_0$, the parity relations at $v_0$ flip, and under $g$, $v_0$ has the same parity as $2j - k$ neighbors. Therefore,
\[
\sum_{i=1}^{2j} g^*(v_0v_i) = \vert (2j - k)- k \vert =\vert  2j - 2k\vert .
\]

Then the net difference in total signed edge weights becomes:
\[
\tilde{\varepsilon}_f - \tilde{\varepsilon}_g =\vert 2k - 2j \vert - \vert 2j - 2k \vert  \equiv 0 \pmod{4}.
\]

Hence, the result follows.
\end{proof}



\begin{thrm}
The circulant graph $\Gamma(n,\pm\{1,2,\ldots,4m+1\})$, where $m \le \frac{n-2}{4}$, is not Fibonacci cordial whenever $n \equiv 2 \pmod{4}$.
\end{thrm}

\begin{proof}
Suppose $n = 4q + 2$ for some positive integer $q$. Then the number of edges in $\Gamma$ is given by 
\[
|E(\Gamma)| = (4q + 2)(4m + 1) = 4(4mq + 2m + q) + 2.
\]
It follows that $|E(\Gamma)| \equiv 2 \pmod{4}$. By the necessary condition established in the preceding theorem, this congruence implies that a Fibonacci cordial labeling is not possible. Hence, the result follows.
\end{proof}


Similar reasoning yields the following result.

\begin{thrm}
The circulant graph $\Gamma(n,\pm\{1,2,\ldots,j\})$ is not Fibonacci cordial in the following cases:
\begin{enumerate}
    \item When $j \equiv 2 \pmod{4}$ and $n \equiv 1 \pmod{2}$,
    \item When $j \equiv 3 \pmod{4}$ and $n \equiv 2 \pmod{4}$.
\end{enumerate}
\end{thrm}

\begin{thrm}
The circulant graph $\Gamma(n,\pm\{1,2\})$ is Fibonacci cordial if and only if $n \not \equiv 1\ (\mod 2)$.
\end{thrm}
\begin{proof}
We need only to prove that $\Gamma(n,\pm\{1,2\})$ is Fibonacci cordial when $n$ is even, i.e., $n \equiv 0 \pmod{2}$. We do so by explicitly constructing a Fibonacci cordial labeling of the vertices of $\Gamma$. For $n\equiv 2 \ (\mod 6)$ we consider the following function that assign even Fibonacci numbers to the vertices. 

First, consider the case when $n \equiv 2 \pmod{6}$. Define a vertex labeling function $f$ as follows:
\[
f(v_i) =
\begin{cases}
F_0, & \text{if } i = 1, \\
F_3, & \text{if } i = 2, \\
F_{3i/2}, & \text{if } i \equiv 0 \pmod{4}, \\
F_{3(i+1)/2}, & \text{if } i \equiv 1 \pmod{4} \text{ and } i > 1,
\end{cases}
\]
for $i \leq \frac{2(n-2)}{3}$. The remaining vertices are labeled arbitrarily with distinct odd Fibonacci numbers to ensure balance in the edge label distribution.

For other even values of $n$, we extend this labeling scheme:

\begin{itemize}
\item If $n \equiv 4 \pmod{6}$, define \( f\left(\frac{2n - 5}{3}\right) = F_{n-1} \).
\item If $n \equiv 0 \pmod{6}$, define \( f\left(\frac{2n}{3}\right) = F_n \).
\end{itemize}

These assignments preserve the necessary distribution of edge labels as derived from Fibonacci values, thereby fulfilling the conditions for a Fibonacci cordial labeling and completing the proof.
\end{proof}
\begin{figure}[h!]
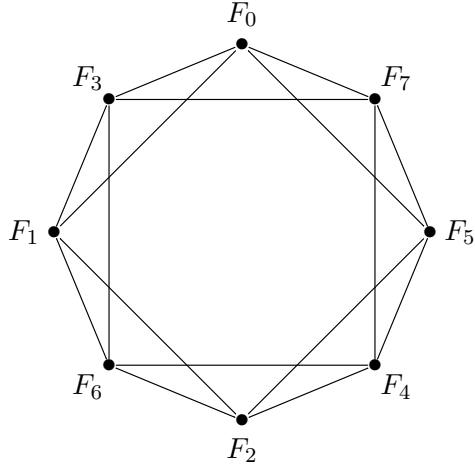

\[
\xygraph{
!{<0cm,0cm>;<0cm,1cm>:<-1cm,0cm>::}
!{(0,0);a(0)**{}?(2.5)}*{\bullet}="v0" !{(0,0);a(0)**{}?(2.9)}*{F_0}
!{(0,0);a(45)**{}?(2.5)}*{\bullet}="v1" !{(0,0);a(45)**{}?(2.9)}*{F_3}
!{(0,0);a(90)**{}?(2.5)}*{\bullet}="v2" !{(0,0);a(90)**{}?(2.9)}*{F_1}
!{(0,0);a(135)**{}?(2.5)}*{\bullet}="v3" !{(0,0);a(135)**{}?(2.9)}*{F_6}
!{(0,0);a(180)**{}?(2.5)}*{\bullet}="v4" !{(0,0);a(180)**{}?(2.9)}*{F_2}
!{(0,0);a(225)**{}?(2.5)}*{\bullet}="v5" !{(0,0);a(225)**{}?(2.9)}*{F_4}
!{(0,0);a(270)**{}?(2.5)}*{\bullet}="v6" !{(0,0);a(270)**{}?(2.9)}*{F_5}
!{(0,0);a(315)**{}?(2.5)}*{\bullet}="v7" !{(0,0);a(315)**{}?(2.9)}*{F_7}
"v0"-"v1" "v1"-"v2" "v2"-"v3" "v3"-"v4" "v4"-"v5" "v5"-"v6" "v6"-"v7" "v7"-"v0"
"v0"-"v2" "v1"-"v3" "v2"-"v4" "v3"-"v5" "v4"-"v6" "v5"-"v7" "v6"-"v0" "v7"-"v1"
}
\]
\caption{Fibonacci cordial labeling of $\Gamma(8,\pm\{1,2\})$.}
\label{circulant8}
\end{figure}


{
\begin{thrm}
The circulant graph $\Gamma(n,\pm\{1,2,3\})$ is Fibonacci cordial if and only if $n = 6$ or $n \equiv 0,1,3 \pmod{4}$.
\end{thrm}

\begin{proof}
First, observe that when $n = 6$, the circulant graph $\Gamma(6,\pm\{1,2,3\})$ is isomorphic to the complete graph $K_6$, which is known to be Fibonacci cordial (see \cite{MitraB2020}). Moreover, computational verification shows that $\Gamma(n,\pm\{1,2,3\})$ is Fibonacci cordial for all $7 \leq n \leq 28$; in particular, $\Gamma(7,\pm\{1,2,3\})$ is also a complete graph and thus Fibonacci cordial.

It remains to construct a Fibonacci cordial labeling for $n \geq 29$ when $n \equiv 0,1,3 \pmod{4}$. We begin with an explicit construction for $n \equiv 1 \pmod{12}$ and then describe how to extend or adapt this labeling to other congruence classes.

Note that we only detail the labeling of a subset of the vertices using even Fibonacci numbers. The remaining vertices can be assigned unused odd Fibonacci numbers in such a way that the edge labels are balanced, ensuring a Fibonacci cordial labeling.

\begin{mycases}
\case \label{case1} \textbf{$n \equiv 1 \pmod{12}$:} For $1 \leq i < 21$, define:
\[
f(v_i) =
\begin{cases}
F_{3i/2}, & \text{if } i = 2, 6, 10, 14, \\
F_{3(i-1)/2}, & \text{if } i = 1, 5, 9, 13, 17.
\end{cases}
\]
For $21 \leq i < \frac{3(n-1)}{4}$, set:
\[
f(v_i) =
\begin{cases}
F_{(4i-3)/3}, & \text{if } i \equiv 3 \pmod{9}, \\
F_{(4i+2)/3}, & \text{if } i \equiv 4 \pmod{9}, \\
F_{(4i-1)/3}, & \text{if } i \equiv 7 \pmod{9}, \\
F_{(4i+5)/3}, & \text{if } i \equiv 8 \pmod{9}.
\end{cases}
\]

\case \textbf{$n \equiv 3 \pmod{12}$:} Extend the labeling from Case~\ref{case1} and define
\[
f\left(\frac{3(n-3)}{4}\right) = F_n.
\]

\case \textbf{$n \equiv 4 \pmod{12}$:} Use the labeling from Case~\ref{case1} and set
\[
f\left(\frac{3n - 4}{4}\right) = F_{n-1}.
\]

\case \textbf{$n \equiv 5 \pmod{12}$:} Use the same labeling as in Case~\ref{case1} for $21 \leq i \leq \frac{3(n-1)}{4}$.

\case \textbf{$n \equiv 7 \pmod{12}$:} Use the same labeling as in Case~\ref{case1} for $21 \leq i < \frac{3(n-7)}{4}$, and then define:
\[
f(v_{\frac{3(n-7)}{4}}) = F_{n-4}, \quad f(v_{\frac{3n - 5}{4}}) = F_{n-1}.
\]

\case \label{case8} \textbf{$n \equiv 8 \pmod{12}$:} Use the same labeling as in Case~\ref{case1} for $21 \leq i \leq \frac{3n - 8}{4}$.

\case \textbf{$n \equiv 9 \pmod{12}$:} Use the labeling from Case~\ref{case8} and define
\[
f(v_{\frac{3n - 7}{4}}) = F_n.
\]

\case \textbf{$n \equiv 11 \pmod{12}$:} Use the same labeling as in Case~\ref{case1} for $21 \leq i \leq \frac{3n - 5}{4}$.

\case \textbf{$n \equiv 0 \pmod{12}$:} Use the labeling from Case~\ref{case1} for $21 \leq i \leq \frac{3n}{4} - 5$, and then set:
\[
f(v_{\frac{3n}{4} - 4}) = F_{n-3}, \quad f(v_{\frac{3n}{4} - 1}) = F_n.
\]
\end{mycases}

This completes the proof that $\Gamma(n,\pm\{1,2,3\})$ is Fibonacci cordial precisely when $n = 6$ or $n \equiv 0,1,3 \pmod{4}$.
\end{proof}

}

\begin{thrm}
The circulant graph $\Gamma(n,\pm\{1,2,3,4\})$ is Fibonacci cordial for all $n > 8$.
\end{thrm}

\begin{proof}
First, note that $\Gamma(9,\pm\{1,2,3,4\})$ is isomorphic to the complete graph $K_9$, which is known to be Fibonacci cordial (see \cite{MitraB2020}). Therefore, it suffices to prove the result for $n \geq 10$.

We first provide an explicit labeling for $n = 10$ as a special case. Let $f$ be the labeling function, then $f(v_i)=F_{i}$ for $0\le i\le 10$. This assignment generates cordial labeling. 
For $n \geq 11$, we present a unified labeling strategy by dividing the cases based on $n \pmod{3}$. In each case, we assign even-indexed Fibonacci numbers to a subset of the vertices. The remaining vertices can be labeled arbitrarily using unused odd Fibonacci numbers to ensure a Fibonacci cordial labeling.

Let $f$ be the labeling function again, and let $1 \leq i \leq k$, where $k$ depends on $n$ as described in each case below.

\begin{mycases}
\case \textbf{$n \equiv 0 \pmod{3}$}

\[
f(v_i) =
\begin{cases}
F_{3(i-1)}, & \text{for } 1 \leq i \leq 4, \\
F_{9}, & \text{for } i = 7, \\
F_{3(2i + 3)/5}, & \text{for } i > 7 \text{ and } i \equiv 1 \pmod{5}, \\
F_{3(2i + 6)/3}, & \text{for } i > 7 \text{ and } i \equiv 2 \pmod{5}.
\end{cases}
\]

The index $k$ is defined as:
\[
k =
\begin{cases}
\frac{5n - 9}{6}, & \text{if } n \equiv 3 \pmod{6}, \\
\frac{5n}{6} - 3, & \text{if } n \equiv 0 \pmod{6}.
\end{cases}
\]

\case \textbf{$n \equiv 1 \pmod{3}$, $n > 10$}

\[
f(v_i) =
\begin{cases}
F_{3(i-1)}, & \text{for } 1 \leq i \leq 3, \\
F_{9}, & \text{for } i = 6, \\
F_{12}, & \text{for } i = 7, \\
F_{3(2i + 3)/5}, & \text{for } i > 7 \text{ and } i \equiv 1 \pmod{5}, \\
F_{3(2i + 6)/3}, & \text{for } i > 7 \text{ and } i \equiv 2 \pmod{5}.
\end{cases}
\]

The index $k$ is given by:
\[
k =
\begin{cases}
\frac{5n - 23}{6}, & \text{if } n \equiv 1 \pmod{6}, \\
\frac{5n - 14}{6}, & \text{if } n \equiv 4 \pmod{6}.
\end{cases}
\]

\case \textbf{$n \equiv 2 \pmod{3}$}

\[
f(v_i) =
\begin{cases}
F_{3(i-1)}, & \text{for } 1 \leq i \leq 3, \\
F_{9}, & \text{for } i = 6, \\
F_{6i/5}, & \text{for } i > 6 \text{ and } i \equiv 0 \pmod{5}, \\
F_{3(2i + 3)/3}, & \text{for } i > 6 \text{ and } i \equiv 1 \pmod{5}.
\end{cases}
\]

The index $k$ is given by:
\[
k =
\begin{cases}
\frac{5n - 10}{6}, & \text{if } n \equiv 2 \pmod{6}, \\
\frac{5n - 19}{6}, & \text{if } n \equiv 5 \pmod{6}.
\end{cases}
\]
\end{mycases}

Thus, in all cases for $n \geq 10$, a Fibonacci cordial labeling exists.
\end{proof}

We conclude this section with an asymptotic conjecture regarding circulant graphs of the form $\Gamma(n,S)$.

\begin{conj} 
For large $n$, almost every circulant graph $\Gamma(n,S)$ with a ``small" connection set $S$ admits a Fibonacci cordial labeling. 
\end{conj}

\section{Conclusion}

In this paper, we explored the existence of Fibonacci cordial labelings for several significant families of graphs, including Generalized Petersen graphs, open and closed helm graphs, joint sum graphs, and circulant graphs of small order. These results contribute to the expanding field of graph labelings motivated by numerical sequences, particularly those derived from the Fibonacci sequence. Potential directions for future research include extending Fibonacci cordiality to larger circulant graphs with smaller connection sets and other Cayley graphs, including various graph products such as the lexicographic and corona products. Another promising direction involves investigating analogous labeling schemes based on other well-known sequences, such as the Tribonacci or Perrin numbers, with the aim of uncovering deeper structural or algebraic relationships within graph labeling theory.

\bibliographystyle{amsplain}

\end{document}